\documentclass[reqno]{amsart}
\newcommand{\E}{{\mathbb E}}
\usepackage[inner=1.0in,outer=1.0in,bottom=1.0in, top=1.0in]{geometry}

\def\XXint#1#2#3{{\setbox0=\hbox{$#1{#2#3}{\int}$ }
\vcenter{\hbox{$#2#3$ }}\kern-.6\wd0}}

\usepackage{color}

\newcommand{\dbar}{\overline{\partial}}

\usepackage{epsfig}
\usepackage{amscd}
\usepackage{amsmath, amssymb, amsthm}
\usepackage{graphics}
\usepackage{color}
\newcommand{\ncal}{{\mathcal N}}

\newcommand{\half}{\frac{1}{2}}
\newcommand{\kahler}{K\"ahler }

\usepackage{caption}
\usepackage{subcaption}

\usepackage{latexsym}

\newcommand{\Ss}{{\mathbb S}}

\newcommand{\PP}{{\mathbb P}}
\newcommand{\R}{{\mathbb R}}
\newcommand{\C}{{\mathbb C}}

\newcommand{\Z}{{\mathbb Z}}

\newcommand{\CP}{\C\PP}

\renewcommand{\phi}{\varphi}

\newcommand{\hcal}{\mathcal{H}}

\newcommand{\ocal}{\mathcal{O}}

\newcommand{\zcal}{\mathcal{Z}}
\newcommand{\jcal}{\mathcal{J}}

\newcommand{\La}{\Lambda}

\newcommand{\De}{\Delta}

\usepackage{mathrsfs, amsmath,amssymb,amsfonts, amsthm}
\newtheorem{theorem}{Theorem}[section]

\newtheorem{prop}[theorem]{Proposition}
\newtheorem{lem}[theorem]{Lemma}

\newtheorem{lemma}[theorem]{Lemma}

\newtheorem{definition}[theorem]{Definition}

\numberwithin{equation}{section}

\title{Topology  of the nodal set of random equivariant spherical harmonics on $\Ss^3$}

\author{Junehyuk Jung }

\address{Department of Mathematics, Texas A\&M University, College Station, TX 77845 USA}

\email{junehyuk@math.tamu.edu}

\author{Steve Zelditch}

\address{Department of Mathematics, Northwestern  University, Evanston, IL 60208, USA}

\email{zelditch@math.northwestern.edu}

\thanks{Research partially supported by NSF grant  DMS-1810747. The first author is partially supported by Sloan Research Fellowship and by NSF grant DMS-1900993.}

\begin{document}

\begin{abstract} We show that real and imaginary parts of equivariant spherical harmonics on $\Ss^3$ have almost
surely a single nodal component. Moreover, if the degree of the spherical harmonic is $N$ and the equivariance degree
is $m$, then the expected genus is proportional to $m \left(\frac{N^2 - m^2}{2} + N\right) $. Hence if $\frac{m}{N}= c $ for fixed $0 < c < 1$, the genus has order $N^3$.
\end{abstract}

\maketitle


\section{Introduction and statements of the results}

In a recent article \cite{JZ18} the authors proved that nodal sets of  real or imaginary parts of equivariant (but  non-invariant)   eigenfunctions of Laplacians $\Delta_{KK}$  of  generic `Kaluza-Klein' metrics  $g_{KK} $ on  unit tangent (or cotangent)  bundles $\pi: M \to X$
over Riemann surfaces $(X, g)$ have a single connected component.  The generic condition is that $0$ be a regular value for
the eigenfunctions.
The unit sphere $\Ss^3 \subset \R^4$ with its standard metric and Laplacian has the standard Hopf fibration $\pi: \Ss^3 \to
\Ss^2$ and  is an example of a Kaluza-Klein  metric. It is a double cover $\Ss^3 \to SO(3) \simeq U(\Ss^2)$ of the unit tangent
bundle of $\Ss^2$.  Recently, the nodal sets of random wave on $3$-dimensional Euclidean space have been the subject of numerical investigations by A. Barnett, Kyle Konrad, and Matthew Jin
\cite{B18},  which
exhibit a surprising feature: only a small number of nodal components is visible in the computer graphics (Figure \ref{Fig1}).
\begin{figure}[h]
\centering
\begin{minipage}{.5\textwidth}
\centering
 \includegraphics[width=.6 \linewidth]{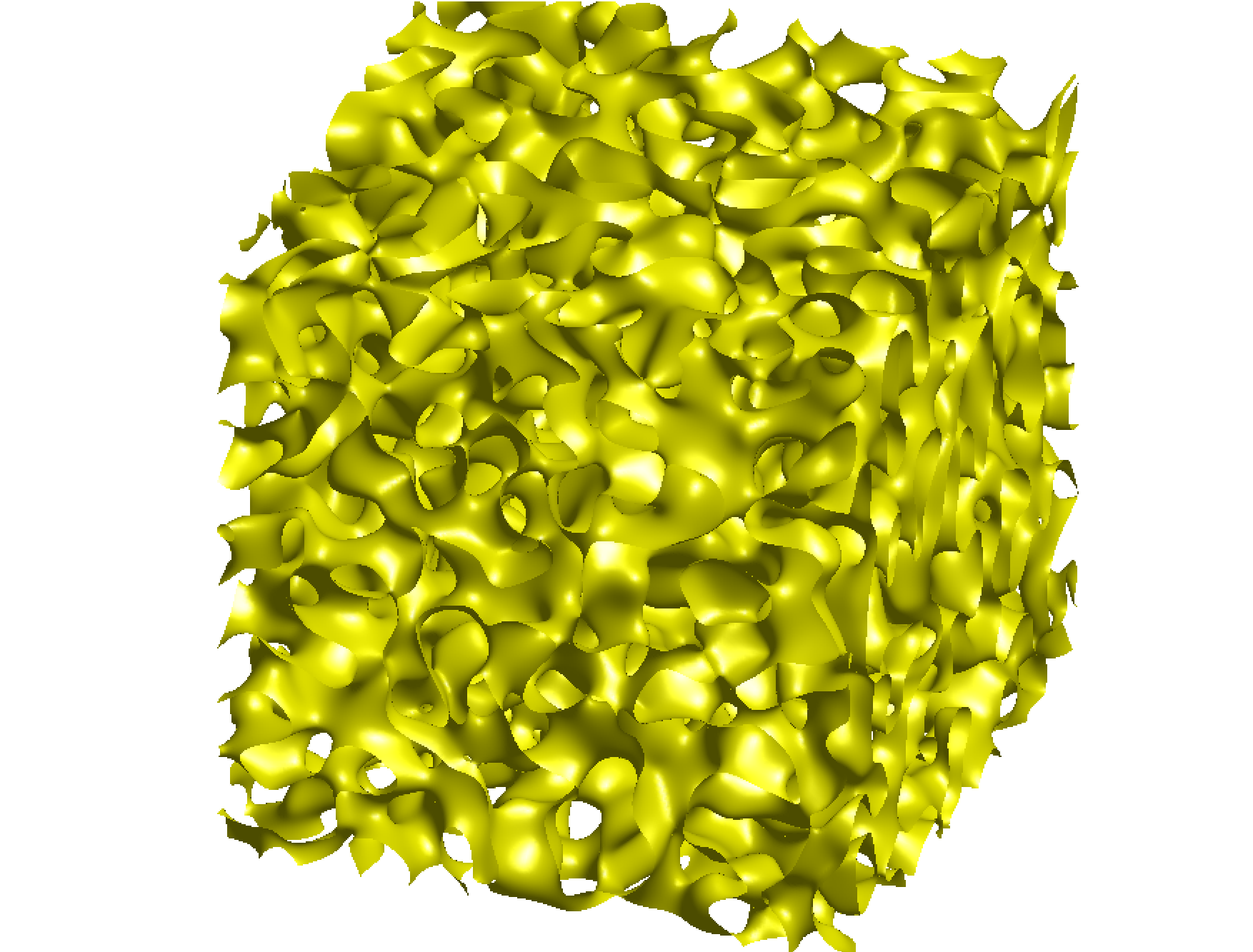}
\captionof{figure}{ Nodal surface  (\cite{B18})}
\label{Fig1}
\end{minipage}%
\begin{minipage}{.5\textwidth}
\centering
\includegraphics[width=.6 \linewidth]{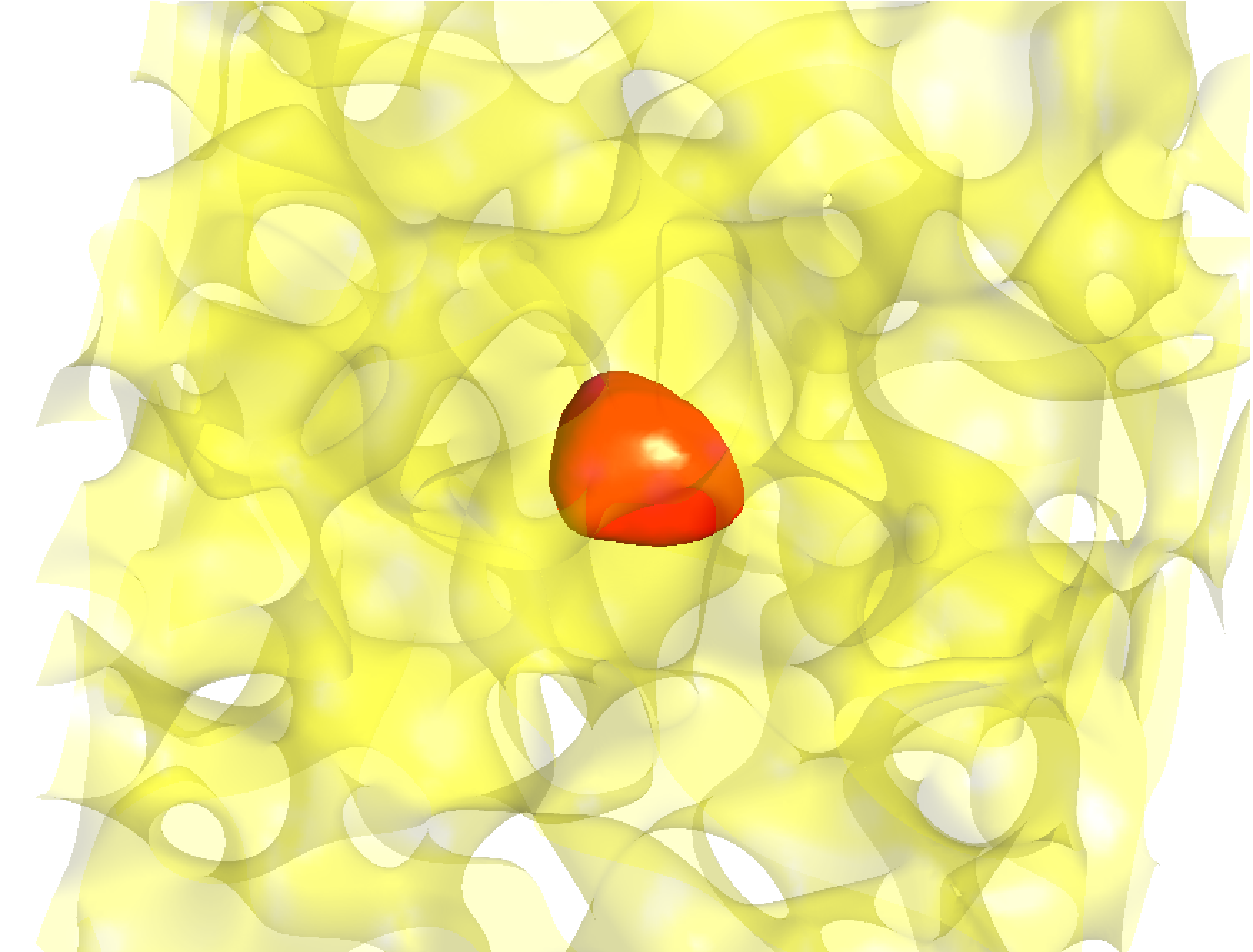}\\
\captionof{figure}{ Nodal surface plus one extra component  (\cite{B18})}
\label{Fig2}
\end{minipage}
\end{figure}

The results of Nazarov-Sodin \cite{NS16}
show that in fact there must be $c N^3$ distinct nodal components for some (very small) $c > 0$, but the other components are evidently too small to be seen in the
computer graphics. For this reason, it is conjectured that with probability one, the random spherical harmonic of fixed degree $N$
on $\Ss^3$ has one {\it giant component} and many much smaller components.  In Figure \ref{Fig2}, one does see a second small (red) component.

 P. Sarnak posed the problem of finding its expected genus, and
has proposed that the expected genus is of the order of magnitude of $N^3$ \cite{Sar18}. Milnor has proved that the maximal genus of the zero
set of a polynomial of degree $N$ is of this order of magnitude, so Sarnak's proposal is that  the nodal sets of random spherical harmonics
on $\Ss^3$ are rather  like Harnack curves (real algebraic curves of degree $N$ and of roughly maximal genus).

The purpose of this note is to link the results of \cite{JZ18} to Sarnak's proposal. Our main result is that the proposal is true
for real/imaginary parts of  random equivariant spherical harmonics of degree $N$ on $\Ss^3$. That is, we let $\hcal_N^m(\Ss^3)$ denote
the space of spherical harmonics of degree $N$ which are equivariant of degree $m$ with respect to the $S^1$ action defining
the Hopf fibration $\Ss^3 \to \Ss^2$. For short, we say that $\psi_N^m$ is equivariant of degree $(N, m)$. We restrict
the natural Gaussian measure on the space $\hcal_N \subset L^2(\Ss^3)$ of spherical harmonics of degree $N$  to the subspaces $\hcal_N^m(\Ss^3)$.  Since
$L^2(\Ss^3) = \bigoplus_{N=0}^{\infty} V_N \otimes V_N^*$ where $V_N$ is the $N$th irreducible representation of $\Ss^3$, fixing
the weight $m$ of the $S^1$ action is the same as fixing one line in $V_N^*$, so that $\dim \hcal_N^m = \dim V_N = N+1 $. We let
$\gamma_N^m$ be the induced Gaussian measure on $\hcal_N^m$.
Its  covariance function is the Schwartz kernel $\Pi_N^m(x, y)$ of the orthogonal projection,
\begin{equation} \Pi_N^m: L^2(\Ss^3) \to \hcal_N^m. \end{equation}

We denote by
\begin{equation} \E_{N, m} \;g (\zcal_{u_N^m}) \end{equation}
the expected genus of the nodal set of the real part $u_N^m$ of an equivariant eigenfunction of degree $(N, m)$ with
respect to $\gamma_N^m$.

\begin{theorem} \label{MAIN}  Let $(\hcal_N^m, \gamma_N^m)$ be the Gaussian space of equivariant spherical harmonics
of degree $(N, m)$. In the following, we assume $m \not= 0$.  \bigskip

\begin{itemize}
\item[(i)] With probability $1$ (w.r.t. $\gamma_N^m$), the nodal set $\zcal_{u_N^m} $  of the   real part $u_N^m = \Re \psi_N^m$ (resp. the nodal set $\zcal_{v_N^m} $ of the  imaginary part $v_N^m = \Im \psi_N^m$) of a random equivariant spherical harmonic $\psi_N^m \in \hcal_N^m$  has a single connected component; \bigskip

\item[(ii)] The expected genus  of the nodal component is by
$$\E_{N,m} g(\zcal_{u_N^m} ) = \frac{1+\eta^2}{8\pi} |m| \left(\frac{N^2 - m^2}{2} + N\right) - |m|+1, $$ where
\[
\eta = \frac{m/2}{\frac{N^2 - m^2}{2} + N},
\]
which has modulus less than or equal to $1/2$.

\end{itemize}
\end{theorem}
\bigskip

The first statement is almost an application of the  main result of \cite{JZ18}, where the `genericity' assumption was only
used to prove that each real (resp. imaginary) part of an equivariant eigenfunction has $0$ as a regular value. Thus, to prove Theorem  \ref{MAIN} (i) it is only necessary to prove  the Bertini-type theorem that $0$ is a regular value of $\psi_N^m \in \hcal_N^m$
with probability $1$. This is done in Section \ref{BERTINISECT}. The second statement has one topological  simple part and one probabilistic part.
The topological makes use of the identification of equivariant functions in $\hcal_N^m$ with sections $f_N^m e_L^m$ of the complex line bundle $\ocal(m) \to \Ss^2$;
here $e_L$ denotes a local frame of $L$ over the affine chart, which we choose to be a holomorphic frame (see Section \ref{LBSECT}).
\begin{lemma} \label{CHILEM} If $0$ is a regular value of $\psi_N^m$, then the Genus of $\zcal_{\Re\psi_N^m}$ is given by
\begin{equation}
\frac{|m|( \#\{f_N^m = 0\}-2)}{2}+1.
\end{equation}
where  $\psi_N^m$ is the lift of the section $f_{N}^m (z, \bar{z}) e_L^m. $
\end{lemma}
As mentioned in  \cite{JZ18}, the key point of the proof is to show that  $\pi : \ncal_{u_{m,j}}  \to X$  is a kind of `helicoid cover'. That is, it is an $m$-fold
cover over the complement of the zeros of $f_N^m$, while the inverse image of a zero is an $S^1$ orbit. Thus, it is not a branched cover in the standard
sense; rather  $\pi$ is locally like the projection of a vertical helicoid onto the horizontal plane.\footnote{We thank J. Y. Welschinger for discussions of this
local picture and for the Euler characteristic calculation of Lemma \ref{CHILEM}.}

The second part is the following Kac-Rice type calculation.
\begin{lemma} \label{KRLEM} With probability $1$, $f_N^m$ has isolated non-degenerate zeros, and
$$\E \#\{f_N^m = 0\} =\frac{1+\eta^2}{4\pi} \left( \frac{N^2 - m^2}{2} + N\right), $$
where
\[
\eta = \frac{m/2}{\frac{N^2 - m^2}{2} + N}.
\]
\end{lemma}

To determine the expected  number of zeros of $f_N^m$ we use the Kac-Rice formula for Gaussian random eigen-sections of a line bundle,
which gives an integral formula for  the expected number of zeros in terms of the determinant of
a matrix formed from values and derivatives of the covariance (two-point) function of the Gaussian random function.

It follows that when $\delta<\frac{m}{N}<1-\delta$ for some small fixed $\delta>0$, the genus of   $\zcal_{u_{N}^{m_N}}$  of order $N^3$. The spherical
harmonics in $\bigcup_{\substack{|m| \leq N\\ 2|N-m}} \hcal_N^m$ form a set of $ N+1 $ subspaces of dimension $N+1$ in $\hcal_N$, which intersect
the unit sphere $S\hcal_N$ in a union of $N+1$ great spheres of dimension $N $. A random spherical harmonic in $\hcal_N$ is
a random linear combination of components in $\hcal_N^m$. It is tempting to imagine that one can determine the expected Euler
characteristic of the nodal set of a  random spherical harmonic in $\hcal_N$ by perturbing the special equivariant spherical harmonics in $\hcal_N^m$
(or more precisely, their real and imaginary parts). The genus  of the nodal set of $u_N^m = \Re \psi_N^m$ is so large
that a small perturbation is unlikely to decrease the genus.  But this is just speculation.

\subsection{Acknowledgements}  We thank Peter Sarnak for discussions on the genus of the giant component which prompted this article. We also thank Jean-Yves Welschinger
for discussions of the Euler characteristic calculation at the outset. We thank Igor Wigman for helpful discussions.

\section{Background and notations}
In this section, we  first review the results of \cite{JZ18}  on eigenfunctions of Kaluza-Klein Laplacians on general $S^1$ bundles over Riemann
surfaces. We  then specialize to the circle bundle $\Ss^3 \to \Ss^2$ and introduce some notation and background concerning eigenfunctions
of the standard Laplacian $\Delta_{\Ss^3} $on $\Ss^3$, on irreducible representations of $SU(2) = \Ss^3$ and on sections of associated complex line bundles.

\subsection{\label{JZREVIEW} Review of the results of \cite{JZ18} }
The  main result is \cite[Theorem 1.5]{JZ18} pertains to nodal sets of real/imaginary parts of equivariant eigenfunctions of the Kaluza-Klein Laplacian on $S^1$ bundles
$\pi: M \to X$ over Riemannian surfaces $X$. A Kaluza-Klein metric $G$ is specified by a metric $h$ on $X$ and a connection $\nabla$ on $\pi: M \to X$.

 \begin{theorem}\label{JZT} Suppose that the data $(g, h, \nabla)$ of
 the Kaluza-Klein metric satisfies the generic properties of \cite[Theorem 1.4]{JZ18}. Then,

 \begin{enumerate}

\item The eigenspace of $\Delta_G$ corresponding to $\lambda=\lambda_{m,j}=\lambda_{-m,j}$ is spanned by $ \phi_{m,j}$ and  $\phi_{-m,j}= \overline{\phi_{m,j}}$. In particular, any real eigenfunction with the eigenvalue $\lambda_{m,j}$ is a constant multiple of $T_\theta \left(\Re \phi_{m,j}\right) $, where $T_\theta$ is the $S^1$ action on $P$ parameterized by $\theta$.

 \item For $m \neq 0$, the nodal sets of $\Re \phi_{m,j}$  are connected.

   \item For $m \neq 0$, the number of  nodal domains of $\Re \phi_{m,j}$  is $2$.
 \end{enumerate}

 \end{theorem}
 The principal `generic property' referred to in the statement is that $0$ is a regular value of the (complex-valued) equivariant eigenfunctions. It is also
 proved that for generic Kaluza-Klein metrics, the eigenspaces are of real dimension $2$ (corresponding to an equivariant eigenfunction and its complex
 conjugate).

 Although the standard metric on $\Ss^3$ is Kaluza-Klein with respect to the standard metric on $\Ss^2$ and Riemannian connection,
 the above theorem does not apply to the setting of the present paper. Indeed, the eigenspace of $\Delta_{\Ss^3}$ corresponding to spherical
 harmonics of degree $N$ is of dimension $N^2$, quite opposite to the `multiplicity two' property for generic Kaluza-Klein metrics. However, this
 property was not used to prove (2)-(3), but only to ensure that for generic K-K metrics, all eigenfunctions are equivariant, i.e. the Laplace eigenvalue
 determines the equivariance degree. This is certainly not true for $\Delta_{\Ss^3}$. In Theorem \ref{MAIN}, however, we restrict to the subspaces
 of Laplace eigenfunctions with a fixed equivariance degree and prove the result for them, as long as $m \not= 0$. We do not allow linear combinations
 of Laplace eigenfunctions of different equivariance degrees.

The one `generic property' we do need is that $0$ is a regular value of the eigenfunctions. This is not necessarily the case for every equivariant eigenfunction,
but we prove it is the case almost surely for random equivariant spherical harmonics (Theorem \ref{BERTINI}). Once this is established, the proof
of   \cite[Theorem 1.5]{JZ18} applies with no change to equivariant spherical harmonics, and allows us to conclude that (2)-(3) hold.
See Section \ref{COMPLETION} for a final summary of the argument.

\subsection{ Coordinates on  $\Ss^3$ and Hopf fibration}\label{coor}
We use two coordinate systems on $\Ss^3$:
\begin{equation}\label{coord1}
(z_1,z_2) \mapsto \begin{pmatrix}\Re z_1\\ \Im z_1\\ \Re z_2 \\ \Im z_2\end{pmatrix}\in \Ss^3 \subset \mathbb{R}^4
\end{equation}
where $z_1,z_2\in \mathbb{C}$ and $|z_1|^2+|z_2|^2=1$, or
\begin{equation}\label{coord2}
(\alpha, \varphi, \theta)\mapsto \begin{pmatrix}\sin\alpha \cos(\theta+\varphi)\\ \sin\alpha \sin(\theta+\varphi) \\ \cos\alpha \cos(\theta-\varphi) \\ \cos\alpha \sin(\theta-\varphi)\end{pmatrix} \in \Ss^3 \subset \mathbb{R}^4,
\end{equation}
where $\alpha \in \lbrack 0,\pi/2 \rbrack$,  $\varphi \in\lbrack 0,2 \pi \rbrack $, and $\theta \in \lbrack -\pi ,\pi \rbrack$. In the first coordinate system \eqref{coord1}, the action of $S^1$ is given by
\begin{equation}\label{fiber}
e^{i\vartheta}. (z_1,z_2) = (e^{i\vartheta}z_1,e^{i\vartheta}z_2)
\end{equation}
and the Hopf map $\pi:\Ss^3 \to \Ss^2$ is given by
\begin{equation}\label{hopf}
\pi:(z_1,z_2) \mapsto (2z_1\overline{z_2}, |z_1|^2-|z_2|^2) \in \mathbb{C}\times \mathbb{R}.
\end{equation}
In the second coordinate system, \eqref{fiber} is equivalent to
\[
(\alpha, \varphi, \theta) \mapsto (\alpha, \varphi, \theta+\vartheta)
\]
and the Hopf map \eqref{hopf} is
\[
\pi:(\alpha, \varphi, \theta) \mapsto \begin{pmatrix} \sin (2\alpha) \cos (2\varphi)\\ \sin (2\alpha) \sin (2\varphi)\\ \cos (2\alpha) \end{pmatrix} \in S^2 \subset \mathbb{R}^3.
\]
\subsection{\label{LBSECT} $\Ss^3$ as a Kaluza-Klein 3-fold}
 The purpose of this section is to introduce the principal notation and terminology,
and to explain how $\Ss^3$ (equipped with its standard metric) is an example of a Kaluza-Klein metric  in the sense of
\cite{JZ18}. We also review the relation between equivariant Laplace  eigenfunctions on $\Ss^3$ and associated eigensections
of line bundles  over  $S^2$.  In particular, we show that the results of \cite{JZ18} are valid for random spherical harmonics
in $\hcal_N^m$.

The 3-manifolds $M^3$  studied in \cite{JZ18} are $S^1$ bundles $M^3 \to X$ over Riemann surfaces $X$, in particular the unit tangent or cotangent bundles.
A Kaluza-Klein metric on $M^3$  is a bundle-metric $g$  defined by a connection $\nabla$  on $T M^3$ and a Riemannian metric $h$ on $X$. The circular
fibers are geodesics of the metric (in particular have a constant length) and the metric on the horizontal spaces is the lift of the
metric on $X$. It is evident that the standard metric on $\Ss^3$ is Kaluza-Klein with respect to the Hopf fibration $\pi: \Ss^3 \to \Ss^2$.
Harmonic analysis on spheres and its relation to  the Hopf fibration $\Ss^3 \to \Ss^2$ is elementary and well-known (see e.g. \cite{F72}), so we only review two  aspects of it: (i) relating eigensections
of the Bochner Laplacians on $L^m$ to equivariant eigenfunctions of $\Delta_{\Ss^3}$; (ii) CR geometry of $\Ss^3$.

 Associated to the principal $S^1$ bundle $\Ss^3 \to \Ss^2$  are the line bundles $L^m = \Ss^3 \times_{\chi^m} \C$ where $\chi(e^{i \theta}) =e^{i \theta} $. For $m = 1$, $L = \ocal(1) \to \CP^1$ in the notation of algebraic geometry \cite{GH}; it  is the spin-bundle, i.e.  the square root of the anti-canonical bundle,   $ K_{\CP}^{-1} = (T^{1,0})^{\half} \CP^1$.
When $m = -1$,  $L^{-1} = K_{\CP^1} \simeq T^{*1,0} \CP^1$,  the canonical bundle. In a standard way, we  view $\Ss^3 \subset L^*$
as the unit bundle with respect to the Fubini-Study metric. The connection in this setting is the Chern connection of the Fubini-Study metric; we refer to \cite{GH,JZ18} for background.

Sections $s$ of $L^m \to \Ss^2$ naturally lift to $L^*$  by
$$\hat{s}(z, \lambda):= \lambda^m(s(z)).$$
Thus,  the restriction of the lift of $s \in C(X, L^m)$   $\Ss^3 $ satisfies
$\hat{s}(r_{\theta} x) = e^{im \theta} \hat{s}(x)$. We
 refer to lifts of sections of $L^m$ as `equivariant' functions on $\Ss^3$ and denote the space of such functions by $\hcal^m$.

 The standard Laplacian $\Delta_{\Ss^3}$ is a  Kaluza-Klein Laplacian, i.e has the form,
\begin{equation*}
\Delta_{\Ss^3} = \Delta_H +  \frac{\partial^2}{\partial \theta^2},\;\;\;  \end{equation*}
where $\Delta_H $ is the horizontal Laplacian. The fact that the fiber Laplacian is $ \frac{\partial^2}{\partial \theta^2}$ reflects
the fact that $S^1$ orbits are geodesics isometric to $\R/2 \pi \Z$.
It is obvious that $[\Delta_{\Ss^3},  \frac{\partial^2}{\partial \theta^2}] = 0.$
Since $S^1$ acts isometrically on $(\Ss^3, G)$ we may decompose into its weight spaces,
\begin{equation*}
L^2(\Ss^3, dV) = \bigoplus_{m \in \Z} \hcal^m,
\end{equation*}
where $\hcal^m = \{F: \Ss^3 \to \C: F(e^{i \theta}. x) = e^{i m \theta} F(x)\}. $ The weight spaces
are $\Delta^H$-invariant, i.e., $\Delta_H: \hcal^m \to \hcal^m$.
\begin{definition}
We define $\hcal_N^m$ to be the subspace of degree $N$ spherical harmonics in $\hcal^m$. We call a function belonging to $\hcal_N^m$ an equivariant spherical harmonic of degree $(N,m)$.
\end{definition}

The lifting map gives  a canonical identification
$\hcal_m \cong L^2(X, L^m). $
The Bochner Laplacian $\nabla_m^* \nabla_m$ corresponds to  the horizontal Laplacian under this identification, i.e. $$ \widehat{\nabla_m^* \nabla_m (f (dz)^m)} =
\Delta_H   \widehat{(f (dz)^m)}.$$

Since $\Ss^3$ is a group, $L^2(\Ss^3) = \bigoplus_{N=0}^{\infty} V_N \otimes V_N$ where $V_N $ is an irreducible representation of $\Ss^3$ of dimension $N+1$.
Moreover, $\Delta |_{V_N \otimes V_N} =  N(N+2) = (N+1)^2 - 1$. 

\subsection{\label{CRSECT} CR structure} Another viewpoint is that $\Ss^3 = \partial B \subset \C^2$, i.e. that $\Ss^3$ is the boundary of
the unit ball, a strictly pseudo-convex domain in $\C^2$.
A defining function for $\Ss^3 \subset \C^2$ is the usual Euclidean distance $r = |Z|$ from the origin. Here,  $Z \in \C^2$.  Let
$\dbar r$ be the associated $(0,1)$ form. The theory of spherical harmonics on $\Ss^3$ has been related to the CR geometry and the representation theory
of $SU(2) = \Ss^3$
in \cite{F72}.

Spherical harmonics on   $\Ss^3$  are restrictions of  homogeneous harmonic polynomials on $\C^2$. In complex coordinates $Z = (z_1, z_2)$, the  Euclidean Laplacian is
$$\Delta_{\R^{4}} =  4 \sum_{j=1}^2 \frac{\partial^2}{\partial z_j \partial \bar{z}_j}. $$

Let $\hcal_{N}^{(p, q)} $ denote the space of harmonic homogeneous polynomials of degree $N$ on $\C^2$ which are
of degree $p$ in $z_j$'s and of degree $q$ in the $\bar{z}_k$'s; $N = p + q$.  Then $\hcal^{(p,q)}$ is an irreducible representation
of $SU(2)$ and the space of all spherical harmonics of degree $N$ admits the decomposition \begin{equation} \label{DECOMP} \hcal_N = \bigoplus_{p + q = N} \hcal_{N}^{(p,q)}. \end{equation}
The representation of $U(2)$ on $\hcal_N^{(p,q)}$ is denoted by $\rho(q, -p)$. One has $\rho(q,-p) |_{SU(2)} = \rho(q', - p') |_{SU(2)} \iff p +q = p' + q', $ and $\dim \rho(q, -p) = p + q + 1$. Hence the decomposition \eqref{DECOMP} is another decomposition of $V_N \otimes V_N^*$ into
irreducibles.


The orbits of the Hopf fibration of the action \eqref{fiber}
define the characteristic directions of the CR manifold and lie in the null space of $\partial r |_{T \Ss^3}$. It follows that
on polynomials $z^p \bar{z}^q$  of type $(p,q)$, $S^1$ acts by $e^{i (p-q) \theta}$.
Thus,  the equivariance degree of $\psi_N^{p,q} \in \hcal_N^{p,q}$ is  $m = p -q$. Since $p + q = N$, $m =  2p - N$, i.e. the data $(p,q)$ is equivalent
to specifying only $p$ or $q$ or $m$. In particular, $S^1$ acts by $e^{i N \theta}$ on the space $\hcal_N^{N,0}$ of
holomorphic polynomials $(p,q) = (N, 0)$.

\begin{prop}\label{dimension}
$\hcal_N^{(p,q)} = \hcal_N^{2p-N}$. In particular, the dimension of $\hcal_N^m$ is $N+1$ if $|m| \leq N$ and $2|N-m$, and $0$ otherwise.
\end{prop}

\subsection{Associated sections of line bundles} It is useful to simultaneously keep in mind the `upstairs' picture of equivariant eigenfunctions of
$\Delta_{\Ss^3}$ and the `downstairs' picture of sections of complex line bundles, as described in Section \ref{coor}.

We denote the eigensection corresponding to $\psi_N^m  \in \hcal_N^m$ as $f_N^m e_L^m$ in a local holomorphic frame $e_L$ of $L$, so that
$f_N^m$ is a locally defined function on $\CP^1$,  i.e. is a function on the affine chart $\C$.
Let
\begin{equation*}
\Re f_N^m = a_N^m(z), \;\; \Im f_N^m = b_N^m(z).
\end{equation*}
Then,
\[
f_N^m(z) e^{-i m\theta} = (a_N^m(z) + i  b_N^m(z))(\cos m \theta - i \sin m \theta),
\]
so that with $\psi_N^m  = u_N^m + i v_N^m$,
 \begin{equation} \label{ubab} \left\{ \begin{array}{l}
 u_N^m =  a_N^m  \cos m \theta  +  b_N^m  \sin m\theta,\\
v_N^m = b_N^m  \cos m \theta  -a_N^m  \sin m \theta.
\end{array} \right. \end{equation}

We denote by $\zcal_{f_N^m }$ the  zero set of the eigensection $f_N^m  e_L^m$ on $X$:
\begin{equation*}
\zcal_{f_N^m } = \{z \in X: f_N^m (z) = 0\}.
\end{equation*}
It is easy to see that the zero set $\zcal_{\psi_N^m  }$ of $\psi_N^m  $ is the inverse image of $\zcal_{f_N^m }$
under the natural projection $\pi$:
\begin{equation*}
\zcal_{\psi_N^m  }  = \pi^{-1} \zcal_{f_N^m }.
\end{equation*}
Usually we study the nodal sets of the real and imaginary parts of the lift, not to be confused with the lifts of the real
and imaginary parts of the local expression $f_N^m $ of the section (since the frame $e_L^m$ must also be taken into account).
 In general, it  is not obvious whether or not  the zero set of $f_N^m $ is discrete in $X$.

 We  denote the nodal sets of the real, resp. imaginary parts, of the lift by  $$\ncal_{ \Re \psi_N^m  }  \{p \in P_h: \Re \psi_N^m  (p) = 0\}, ~\mathrm{ resp. }~  \ncal_{\Im \psi_N^m  } = \{p \in P_h: \Im \psi_N^m  (p) = 0\}.$$
The analysis is the same for real and imaginary parts and we generally work with the imaginary part, following the tradition for quadratic differentials.

Our focus is on the nodal sets of the real or imaginary parts of
\begin{equation}\label{uv}
\psi_N^m  = u_N^m  + i v_N^m.
\end{equation}
Since $\Delta_{\Ss^3}$ is a real operator, the real  and imaginary parts \eqref{uv} satisfied the modified eigenvalue system,
\begin{equation*}
\left\{ \begin{array}{l}   \Delta_G  u_N^m  =  -\lambda_N^m   u_N^m , \\ \\
   \Delta_G  v_N^m  = -\lambda_N^m  v_N^m , \\ \\
   \frac{\partial}{\partial \theta} u_j = m v_j, \;\; \frac{\partial}{\partial \theta}  v_j = -m u_j.
   \end{array} \right.
\end{equation*}




\section{Genus of the nodal set: Proof of Lemma \ref{CHILEM} }
In  this section, we relate the genus of the nodal set of $\Re \psi_N^m $ to the number of zeros of $f_N^m$, under the assumption of all zeros of $f_N^m$ being regular (Lemma \ref{CHILEM}). We first recall a few lemmata:
\begin{lemma}\label{lem1}
Let $X$ be a topological space and let $A,B$ be topological subspaces whose interior cover $X$. Then
\[
\chi(X) = \chi(A)+\chi(B) - \chi(A\cap B),
\]
where $\chi(\cdot)$ is the Euler characteristic of $\cdot$.
\end{lemma}
\begin{lemma}\label{lem2}
Let $X$ be a $n$-covering of $M$. Then we have
\[
\chi(X) = n\chi(M).
\]
\end{lemma}
Now we are ready to prove Lemma \ref{CHILEM}:
\begin{proof}
To simplify the notation, let $\psi = \psi_N^m $ and $f=f_N^m$. Let $\{z_j\}_{j=1,2,\ldots,k}$ be the complete set of zeros of $f$. We first note that
\[
Z_{\Re\psi} - \cup_{j=1}^k \{ (z_j,\theta)~:~ \theta \in \lbrack 0,2\pi \rbrack\}
\]
is $m$-covering of $k$-punctured sphere. So the Euler characteristic of
\[
Z_{\Re\psi} - \cup_{j=1}^k \{ (z_j,\theta)~:~ \theta \in \lbrack 0,2\pi \rbrack\}
\]
is $m(2-k)$ by Lemma \ref{lem2}.

Now we apply Lemma \ref{lem1} with $X= Z_{\Re\psi}$, $A=Z_{\Re\psi} - \cup_{j=1}^k \{ (z_j,\theta)~:~ \theta \in \lbrack 0,2\pi \rbrack\}$, and $B$ equal to a sufficiently small open neighborhood of $\bigcup_{j=1}^k \{ (z_j,\theta)~:~ \theta \in \lbrack 0,2\pi \rbrack\}$ in $Z_{\Re\psi}$. Then $B$ is homotopic to $\cup_{j=1}^k \{ (z_j,\theta)~:~ \theta \in \lbrack 0,2\pi \rbrack\}$, which has Euler characteristic equal to $0$, and $A \cap B$ is $m$-covering of a disjoint union of punctured discs, which also has Euler characteristic equal to $0$. This implies that $X=Z_{\Re\psi}$ has Euler characteristic $m(2-k)$, and therefore the conclusion follows.
\end{proof}

We note that $c_1(L^m) = m$ (first Chern number, the integral of the first Chern class). By the Hopf theorem $c_1(L^m)$  is the sum over zeros of a smooth section $f_{N}^m$ with non-degenerate zeros of the index
in $\Z_2$ of the zero. The index is the degree of the locally defined map $\frac{s(z)}{|s(z)|}$ from a small circle centered at the
zero to $\C$ in a local trivialization (\cite[Theorem 11.17]{BT}-\cite[Proposition 12.8]{BT}.) In particular if $s$ is a holomorphic section,
then the indices are all equal to $1$ and $s$ has precisely $m$ zeros (counted with multiplicity). This only occurs in the case $m = N$. Otherwise,
the sections in $\hcal_N^m$ are smooth eigensections, and they have more than $m$ zeros on average, as is shown in the next section.

\section{ Gaussian random equivariant spherical harmonics }

The space $\hcal_N^m$ has thus been identified as the space,
$$ \hcal_N^{p,q} = \left\{\sum_{|\alpha| = p, |\beta | = q}^N c^{p,q}_{\alpha, \beta} z^{\alpha} \bar{z}^{\beta}, \;\; c_{\alpha, \beta}^{p,q} \in \C \right\}, \;\; 2p -N = m. $$
The basis elements $z^{\alpha} \bar{z}^{\beta}$ are orthogonal on $\Ss^3$ but are not of norm $1$. To compute the norms it is advantageous to relate integrals over $\Ss^3$ with Gaussian integrals over $\C^2$, i.e. to use the measure $e^{- |Z|^2} dL(Z)$ where $dL$ is Lebesgue
measure. The calculations are done in \cite[Page 98]{F72} and one finds that (using multi-index notation $z^{\alpha} = z_1^{\alpha_1} z_2^{\alpha_2}$),
$$\int_{\Ss^3} |z^{\alpha}|^2 dV = \frac{2 \pi  \alpha!}{(|\alpha| + 2)!}. $$ Henceforth we denote the orthonormalized monomials
by $\hat{s}_{N, \alpha, \bar{\beta}}: = \frac{z^{\alpha} \bar{z}^{\beta}}{||z^{\alpha} \bar{z}^{\beta}||}$.

\begin{definition} The Gaussian random equivariant spherical harmonic   $\psi_N^m \in \hcal_N^m = \hcal_N^{p,q}, p - q = m$ is defined by the series
$$\psi_N^m(Z)  = \sum_{|\alpha| = p, |\beta | = q}^N a_{\alpha, \beta} \hat{s}_{N, \alpha, \bar{\beta}}, $$
where the coefficients $a_{\alpha, \beta}$ are independent complex normal Gaussians.

\end{definition}
When $p = m = N$  these polynomials are known as the $SU(2)$ Gaussian holomorphic polynomials; for general $(p,q)$ we call them $SU(2)^{p,q}$ polynomials. They
are sometimes called `poly-analytic' functions.
  The subspace $\hcal_N^0$  consists  of invariant eigenfunction pulled back  from $\Ss^2$.
They are real valued and so the equivariant nodal sets are not intersections of two real nodal sets.

\subsection{Covariance kernel}
The covariance kernel of the Gaussian random equivariant spherical harmonics is the  orthogonal projection, $\Pi_N^{m} : L^2(\Ss^3) \to \hcal_N^{m}$, given by
\[
\Pi_N^m (x,y) := \frac{1}{4\pi^2} \iint \exp (-im\theta_1 + im\theta_2) \Pi_N (r_{\theta_1}x, r_{\theta_2}y) d\theta_1 d\theta_2,
\]
where $\Pi_N : L^2(\Ss^3) \to \hcal_N$ is the covariance kernel of the Gaussian random spherical harmonics. 

Explicit calculations to follow are based on the identity,
\begin{equation}\label{PiU}
\Pi_N (x,y) = U_N( x\cdot y ),
\end{equation}
where
\[
U_N (\cos \theta) = \frac{\sin (N+1) \theta}{ \sin \theta} = \frac{e^{i(N+1)\theta}-e^{-i(N+1)\theta}}{e^{i\theta}-e^{-i\theta}}
= e^{iN\theta} + e^{i(N-2)\theta}+ \ldots + e^{-iN\theta}.
\]
is the Chebyshev polynomial of the second kind. (Note: Gegenbauer polynomial for $\dim=3$ is $U_N$.)

\subsection{Real versus complex Gaussian ensembles}

The purpose of this section is to show that that the real parts of  complex Gaussian $SU(2)^{p,q}$  are real Gaussian spherical
harmonics in the standard sense employed by Nazarov-Sodin, Sarnak etc.   (
Taking the real part defines the map,
$$\Re: \psi_N^m  \in (\hcal_N^{m}, \gamma_{N}^{m}) \to \Re \psi_N^m  \in \hcal_N^{m} + \hcal_N^{-m}. $$ Its image is a real subspace we denote by $\hcal_{N,m} \subset \hcal_N$.  If we push forward the complex Gaussian measure $\gamma_N^{m}$ on $\hcal_N^{m}$ we get a Gaussian measure on $\hcal_{N, m}$.
Our claim is that this measure coincides with the real Gaussian measure on $\hcal_N$ conditioned on $\hcal_{N, m}$. We denote by $E_N^m$ the expectation
with respect to $\gamma_N^m$ and $E_{N, m}$ the conditional expectation of $\gamma_N$ conditioned on $\hcal_{N,m}$./

\begin{lemma}   $\E_N^{p,q} \zcal_{\Re \psi_N^m } = \E_{N, m} \zcal_{\Re \psi_N^m } $, and  $\E_N^{p,q} g(\zcal_{\Re \psi_N^m }) = \E_{N, m} g(\zcal_{\Re \psi_N^m })$.\end{lemma}

\begin{proof}

We assume $p + q= N, p-q = m$. The  real part of a complex combination $$\sum_{|\alpha| = p, |\beta| = q} a_{\alpha, \beta} \hat{s}_N^{\alpha, \beta} = \sum_{\alpha, \beta} [ A_{\alpha, \beta} + i B_{\alpha, \beta}] [u_N^{\alpha, \beta} + i v_N^{\alpha, \beta}] $$
equals  $$\sum_{|\alpha| =p, |\beta | = q} [ A_{\alpha, \beta} u_N^{\alpha, \beta} -  B_{\alpha, \beta} v_N^{\alpha, \beta}]. $$

Here, $A_{p,q}, B_{p,q}$ are independent $N(0,1)$ random variables. If we condition on $B_{\alpha, \beta} = 0$ we get the conditional real Gaussian
ensemble.

We consider the measure-valued random variable $\zcal_{\Re \psi_N^m }$ or its Euler characteristic (or any functional of $\Re \psi_N^m )$  as a function $F(A - B)$ where $A, B$ are independent $N(0,1)$ vectors. That is, the Gaussian
measure on the coefficients $A,B$ is the product $d \gamma(A) d \gamma(B)$. The conditional Gaussian ensemble is $d \gamma(A) \delta_0(B)$.  So the Lemma
boils down to proving that
$$\int_{\R^N} \int_{\R^N} F(A - B) d \gamma(A) d \gamma(B) = \int_{\R^N} F(A) d\gamma(A). $$
Since $\gamma * \gamma = \gamma$, the left side equals
$$\int_{\R^N} \int_{\R^N} F(C) d \gamma(A) d \gamma(C - A) = \int_{\R^N} F(C) d \gamma * \gamma (C) = \int_{\R^N} F(C) d\gamma, $$
as claimed.
\end{proof}

\section{\label{KRFSECT} Kac-Rice formula}

In view of Lemma \ref{CHILEM}, the proof of Theorem \ref{MAIN} is reduced to the calculation of $\E_{N, m} \# \{f_N^m = 0\}$.
To this end, we use the Kac-Rice formula. The Kac-Rice formula in the setting of random smooth sections of  complex line bundles over \kahler manifolds is proved
in \cite{DSZ}. It makes use of the canonical lift of sections to equivariant functions on the associated $S^1$ bundle, and is
therefore well adapted to our setting. We briefly review the formula and then move on to the calculation of  $\E_{N, m} \# \{f_N^m = 0\}$.

\subsection{Review of the Kac-Rice formula in our setting}
We closely follow the exposition in \cite[Section 5.4]{DSZ}.

To state the Kac-Rice result, we need some notation and background.
The pullback of the Dirac mass $\delta_0$ at $0 \in \C$ is given by, $$\delta_0 (f) = \sum_{x: f = 0}
\frac{\delta_x}{ |d f \wedge d \bar{f}|}. $$
Here, we use that the Jacobian $J_f$ of $f: \C \to \C$ is given by $J_f = |d f \wedge d \bar{f}|$. The equivarlant lift of $f$ to the circle bundle is $f e^{- m \phi/2}$ where $\phi$ is the
\kahler potential of the Fubini-Study metric. The pullback of $\delta_0$ under this complex-valued function is,
$$\delta_0 (e^{-m \phi/2} f) = \sum_{x: e^{- m \phi/2} f = 0} \frac{\delta_x}{J_{e^{- m \phi/2} f}}   = e^{m \phi} \sum_{x: f = 0}
\frac{\delta_x}{J_{ f}}, $$
since, $D e^{- m \phi/2i} f  = e^{- m \phi/2} Df$ at a zero.

Second, we need to recall the joint probability density $ D_N^m(x, \xi; z)$ of the random variables $$X_z(f_N^m): = f_N^m(z), \Xi_z(f^N_m) := d f_N^m (z)$$ and in particular the `conditional' density
$D(0, \xi; z)$. The joint probability density is given by,

\begin{equation}\label{fg2}D(x,\xi;z)=
\frac{\exp\langle-\Delta^{-1}v,
v\rangle}{\pi^{3}\det\Delta}\;,\qquad v=\begin{pmatrix} x\\
\xi\end{pmatrix}\,,
\end{equation} where  $\Delta = \Delta_N^m(z)$  is the covariance matrix of  $(X_z, \Xi_z)$,
\begin{eqnarray} \label{DELTA}
\Delta_N^m(z)&=&\left(
\begin{array}{cc}
A^m_N & B^m_N \\
B^{m*}_N & C^m_N
\end{array}\right)\,,\nonumber \\
\big( A^m_N\big) &=&
\E\big ( X_z \bar
X_z\big)=\frac{1}{d_N^m}\Pi_N^m(z,0;z,0)\,,\nonumber\\
\big( B^m_N\big)&=&
\E\big( X_z \overline \Xi_{z}\big)= \frac{1}{d_N^m}
\overline{\nabla}^2\Pi_N^m(z,0;z,0)\,,\nonumber\\
\big( C^N\big)&=&
\E\big(  \Xi_{z}\overline \Xi_{z}\big)=
 \frac{1}{d_N^m}
\nabla^1\overline{\nabla}^2\Pi_N^m(z,0;z,0)\,,\nonumber\\
\end{eqnarray}
Here, $\nabla^1_z$, respectively $\nabla^2_z$, denotes
the differential operator on $X\times X$ given by applying
$\nabla_z$ to the first, respectively second, factor. For notational simplicity, we often drop the super- and sub-scripts $(N, m)$ in what follows.

As discussed in \cite{DSZ},   and as is well-known, $D(0,\xi;z)$ is then given by,
\begin{equation}\label{fg4}
D(0,\xi;z)=Z(z) D_{\La}(\xi;z),
\end{equation}
where

\begin{equation}\label{fg5}
D_{\La}(\xi;z)=\frac{1}{\pi^{2}\det\La}\exp\left(
-{\langle\La^{-1}\xi,\xi\rangle}\right)
\end{equation}
is the Gaussian density with covariance matrix
\begin{equation}\label{fg6}
\La=C-B^*A^{-1}B,
\end{equation}
and where  \begin{equation}\label{fg7}
Z(z)=\frac{\det\Lambda}{\pi \det\De}
=\frac{1}{\pi \det  A}.
\end{equation}
The formula \eqref{fg4} for $D(0, \xi; z)$ simplifies to,  \begin{equation}\label{J}  D(0, \xi; z)  = \frac{1}{\pi^3\det A \det
\Lambda} e^{- \langle  \Lambda^{-1} \xi, \xi \rangle }. \end{equation}

\begin{prop} \label{KRPROP} Let $s = f e$ in a local frame and let $\hat{s} = f e^{- m \phi/2}$ . Then, $\E \zcal_s$ is the measure on $\CP^1$
given by  $$\begin{array}{lll}  \E(\zcal_{s }) & = &    \int_{\C^2} \left| \xi \wedge \overline{\xi}\right| \; D(0, \xi: x)  d L( \xi), \end{array}$$
where  $dL$ is Lebesgue measure and where $D(x, \xi; z) $ is the joint probability density of $(f(z), d f(z))$, given by \eqref{J}.
\end{prop}


\begin{proof}

 By definition,
$$\begin{array}{lll}  \E(\langle \zcal_{s },\psi\rangle ) & = &   \E\int_{\CP^1}
\psi(z)\delta_0(fe^{-m\phi/2}) \left| d(f(z,\bar z)e^{-m\phi/2})   \wedge d\bar(f(z,\bar z)e^{-m\phi/2})\right| \\&&\\
& = &  \E\int_{\CP^1}
\psi(z)\delta_0(f) \left| d(f(z,\bar z))   \wedge d\bar(f(z,\bar z))\right|
. \end{array}$$
Here,
 $\left| d(f(z,\bar z)e^{-n\phi/2})   \wedge d\bar(f(z,\bar z)e^{-n\phi/2}) \right|$ is a density (the absolute value of a volume form).
We then replace the $\delta_0(f)$ by the Fourier integral, to get
$$\E(\langle \zcal_{s },\psi  \rangle)  = \int_{\C} \psi(z) \E (e^{i \Re f(z)  \overline{t} } \left| d f \wedge d \bar{f}\right| ) dL(t). $$
In general, if $F$ is a complex Gaussian random field, and $\Phi$ is a (possibly nonlinear) functional,
$$\E (\Phi(F(z), dF(z) ) = \int_{\C \times \C^2} \Phi(z, \xi) {\mathbb P}(F(z) = x, dF(z) = \xi) =
 \int_{\C \times \C^2} \Phi(z, \xi) D(x, \xi; z) dz d \xi. $$
 Therefore,
 $$  \E (e^{i \Re f(z)  \overline{t} } \left| d f \wedge d \bar{f} \right| ) = \int_{\C} \int_{\C^2} e^{i \Re x \bar{t}} \left| \xi \wedge \overline{\xi} \right| D(x, \xi; z) dz d \xi. $$
 Using that $\int_{\C}  e^{i \Re x \bar{t}} d L(t)= \delta_0(x)$ we get
 $$ \int_{\C} \E (e^{i \Re f(z)  \overline{t} } \left| d f \wedge d \bar{f}\right| ) dL(t) =  \int_{\C^2} \left| \xi \wedge \overline{\xi}\right| \; D(0, \xi; z) dL(z) d L( \xi). $$

It remains to compute  $D(x, \xi; z)$, and we outline the calculation in \cite{DSZ} using the real linear  1-jet map $\jcal:=J^1_z$, which is
locally written in terms of an orthonormal basis $\{f_j\}$  as
\begin{equation} \jcal (a) = (x, \xi) := ( \sum_j a_j  f_j(z), \sum
a_j D  f_j(z)), \;\; J_z (a) = \sum
a_j D  f_j(z)) \end{equation}

We may regard
$\jcal$ as a map from $a \in \C^{N}$ into $(x, \xi) \in \C \times \C^2. $  The joint probability density is the push forward of the measure
$e^{-|a|^2/2} da$ under $J^1_z $, $$\jcal_* e^{-|a|^2} da =  D(x, \xi; z ) dL(x)  dL(\xi)
 \;\; \mbox{i.e.}\; D(x, \xi; z)  = \int_{\jcal^{-1}(x, \xi)}
e^{-|a|^2/2} d\dot{a}
$$ where $d \dot{a}$ is the surface Lebesgue measure on the subspace $\jcal^{-1}(x, \xi)$ .
This follows from general principles on
pushing forward complex Gaussians under complex linear maps $F:
\C^d \to \C^n$, whereby

\begin{equation} \label{JCALJCAL} F_* e^{-|a|^2} da = \gamma_{F F^*}, \;\; {\rm i.e.}\;\; \jcal (x, \xi)   = \frac{1}{\det \jcal
\jcal^*} e^{-  \langle [\jcal  \jcal^*]^{-1} (x, \xi), (x, \xi) \rangle }.
\end{equation}
As in \cite{DSZ}, one shows that $\jcal \jcal^* = \Delta_N^m$ above. Conditioning on $x =0$ then gives \eqref{fg4}.

\end{proof}

\subsection{Symmetries and application to Lemma \ref{KRLEM}}
Note that $\Pi_N(\Phi(x),\Phi(y))=\Pi_N(x,y)$ for any isometry $\Phi \in SO(4)$. However, this is not true for $\Pi_N^m(x,y)$. To understand the symmetries of $\Pi_N^m(x,y)$, we first identify $(z_1,z_2)\in \mathbb{C}^2$ in \eqref{coord1} with unit quaternion $p=z_1+z_2 j$. Then the action $r_\theta$ \eqref{fiber} is equivalent to the left multiplication by $e^{i\theta}$.

We recall that the map from $SU(2)\times SU(2)$ to $SO(4)$ given by mapping a pair of unit quaternions $(p,q)$ to the map $\Phi_{p,q}: x \mapsto \bar{p}xq$ is a surjective homomorphism with the kernel $\{(1,1),(-1,-1)\}$. Observe that among $\Phi_{p,q}$, $\Phi_{1,q}$ for any fixed $q$ commutes with the action $r_\theta$, by the associativity of multiplication of quaternions. Hence $\Phi_{1,q}$ leaves $\Pi_N^m(x,y)$ invariant, because
\[
\Pi_N (r_{\theta_1}\Phi_{1,q}(x), r_{\theta_2}\Phi_{1,q}(y)) = \Pi_N (\Phi_{1,q}(r_{\theta_1}x), \Phi_{1,q}(r_{\theta_2}y)) = \Pi_N (r_{\theta_1}x, r_{\theta_2}y).
\]
We also infer that $\Phi_{1,q}$ is well-defined on the fibers of the Hopf fibration $\pi:\Ss^3 \to \Ss^2$, and therefore induces $SU(2)$-action on $\Ss^2$. This induces a surjective homomorphism $SU(2) \to SO(3)$ with the kernel $\{1,-1\}$, and it is well-known that $SO(3)$ acts doubly transitively on $\Ss^2$.

Now to prove Lemma \ref{KRLEM}, we need to calculate $\Delta_N^m$ and $\Lambda$ in \eqref{fg5}. Equivalently, for a local orthonormal frame $E=\{\partial_\theta, \mathbf{e}_1,\mathbf{e}_2\}$ near $x\in \Ss^3$, we need to calculate
\[
\Pi_N^m(x,x),~\begin{pmatrix}\mathbf{e}_1^y\Pi_N^m (x,y)\big|_{x=y} & \mathbf{e}_2^y\Pi_N^m (x,y)\big|_{x=y}\end{pmatrix}
\]
and
\[
\begin{pmatrix}\mathbf{e}_1^x\mathbf{e}_1^y\Pi_N^m (x,y) \big|_{x=y} &\mathbf{e}_1^x \mathbf{e}_2^y\Pi_N^m (x,y)\big|_{x=y}\\ \mathbf{e}_2^x\mathbf{e}_1^y\Pi_N^m (x,y)\big|_{x=y} & \mathbf{e}_2^x\mathbf{e}_2^y\Pi_N^m (x,y)\big|_{x=y}\end{pmatrix}.
\]
We then deduce from $\Phi_{1,q}$-invariance of $\Pi_N^m(x,y)$, and the discussion above that these quantities do not depend on the choice of $x \in \Ss^3$ and $E=\{\partial_\theta, \mathbf{e}_1,\mathbf{e}_2\}$:
\begin{lem}\label{inv}
$\Delta_N^m$ and $\Lambda$ in \eqref{fg5} are constant matrices.
\end{lem}
\subsection{\label{CHEBSECT} Chebyshev calculations}
In this section, we compute $\Delta_N^m$ and $\Lambda$ in \eqref{fg5} explicitly. Firstly, from Lemma \ref{inv}, it is sufficient to compute the matrix at $x=(\alpha,\varphi,\theta) = (\pi/4,0,0)$ in the coordinate system \eqref{coord2}.

Note that $\Pi_N^m(x,y)=0$, if $2 \nmid N-m$ or if $|m|>N$, by Proposition \ref{dimension}. So we assume that $2|N-m$ and $|m| \leq N$ in this section.
\begin{lemma}\label{lem3}
Let $U_N(x)$ be the Chebyshev polynomial of the second kind. Assume that $m\in \mathbb{Z}$ and $N\in \mathbb{N}$ satisfies $|m|\leq N$.

Then we have
\begin{align*}
\frac{1}{2\pi} \int \cos (m\theta)  U_N (\cos\theta) d\theta&=\left\{ \begin{array}{cl} 1 & \text{if } 2|N-m \\ 0 &\text{otherwise.}\end{array}\right.\\
\frac{1}{2\pi} \int \cos (m\theta)  U_N'(\cos\theta) \cos\theta d\theta &= \left\{ \begin{array}{cl} \frac{N^2-m^2}{2}+N & \text{if } 2|N-m \\ 0 &\text{otherwise.}\end{array}\right.
\end{align*}
\end{lemma}
\begin{proof}
Recall that
\[
U_N (\cos \theta) = \frac{\sin (N+1) \theta}{ \sin \theta} = \frac{e^{i(N+1)\theta}-e^{-i(N+1)\theta}}{e^{i\theta}-e^{-i\theta}}
= e^{iN\theta} + e^{i(N-2)\theta}+ \ldots + e^{-iN\theta}.
\]
Therefore the first integral is $1$ if $m$ is equal to one of $N,N-2,\ldots, -N$, and $0$ otherwise. To compute the second integral, we first differentiate the above equation to get:
\begin{multline*}
-U_N' (\cos \theta)\sin \theta =  iNe^{iN\theta} + i(N-2)e^{i(N-2)\theta}+ \ldots + i(-N)e^{-iN\theta} \\
=iN(e^{iN\theta} - e^{-iN\theta})+ i(N-2)(e^{i(N-2)\theta} -e^{-i(N-2)\theta} ) + \ldots
\end{multline*}
and therefore
\[
U_N' (\cos \theta) = 2N(e^{i(N-1)\theta}+ e^{i(N-3)\theta}+\ldots+e^{-i(N-1)\theta}) + 2(N-2) (e^{i(N-3)\theta}+ e^{i(N-5)\theta}+\ldots+e^{-i(N-3)\theta}) + \ldots.
\]
Because the second integral does not depend on the sign of $m$, we assume without loss of generality that $m\geq 0$. Since $\cos(m\theta)\cos\theta = \frac{1}{2}\left(\cos ((m+1)\theta)+\cos((m-1)\theta)\right)$, we have
\[
D_N^m =  \left[N+(N-2)+ \ldots + (m+2) \right] +  \left[N+(N-2)+ \ldots + m \right] = \frac{1}{2}(N^2-m^2) +N
\]
when $2|N-m$, and $0$ otherwise.
\end{proof}

\begin{theorem}\label{cheb}
We have
\[
\Delta_N^m  =\frac{1}{N+1} \begin{pmatrix} 1 & 0 & 0 \\ 0 &\frac{N^2-m^2}{2}+N & \frac{im}{2} \\0 & -\frac{im}{2} & \frac{N^2-m^2}{2}+N \end{pmatrix},
\]
and
\[
\Lambda = \frac{1}{N+1}\begin{pmatrix}\frac{N^2-m^2}{2}+N & \frac{im}{2} \\ -\frac{im}{2} & \frac{N^2-m^2}{2}+N\end{pmatrix}.
\]
\end{theorem}
\begin{proof}
Firstly, we have
\begin{align*}
\Pi_N^m (x,x) &= \frac{1}{4\pi^2} \iint \exp (-im\theta_1 + im\theta_2) \Pi_N (r_{\theta_1}x, r_{\theta_2}x) d\theta_1 d\theta_2\\
&=\frac{1}{4\pi^2} \iint \exp (-im\theta_1 + im\theta_2)  U_N( \cos (\theta_1-\theta_2) ) d\theta_1 d\theta_2\\
&=\frac{1}{2\pi} \int \exp (-im\theta)  U_N( \cos \theta ) d\theta\\
&=\frac{1}{2\pi} \int \cos (m\theta)  U_N( \cos \theta ) d\theta =1,
\end{align*}
by Lemma \ref{lem3}. For $\nu=\alpha$ or $\varphi$, we have
\begin{align*}
\partial_{\nu(y)}\Pi_N^m (x,y)\big|_{x=y} &= \frac{1}{4\pi^2} \partial_{\nu(y)}\iint \exp (-im\theta_1 + im\theta_2) \Pi_N (r_{\theta_1}x, r_{\theta_2}y) d\theta_1 d\theta_2\big|_{x=y}\\
&=\frac{1}{4\pi^2} \iint \exp (-im\theta_1 + im\theta_2) U_N'(\cos(\theta_1-\theta_2)) \partial_{\nu(y)}(r_{\theta_1}x \cdot r_{\theta_2}y)\big|_{x=y} d\theta_1 d\theta_2.
\end{align*}
If we write
\[
r_{\theta_1}x=(\sin\alpha(x) \cos(\varphi(x)+\theta_1), \sin\alpha(x) \sin(\varphi(x)+\theta_1) , \cos\alpha(x) \cos(\theta_1-\varphi(x)) , \cos\alpha(x) \sin(\theta_1-\varphi(x)))
\]
and
\[
r_{\theta_2}y=(\sin\alpha(y) \cos(\varphi(y)+\theta_2), \sin\alpha(y) \sin(\varphi(y)+\theta_2) , \cos\alpha(y) \cos(\theta_2-\varphi(y)) , \cos\alpha(y) \sin(\theta_2-\varphi(y))),
\]
then
\begin{multline*}
\partial_{\alpha(y)}(r_{\theta_1}x \cdot r_{\theta_2}y)\big|_{x=y=(\alpha,\varphi,0)}
=\sin\alpha\cos\alpha \cos(\varphi+\theta_1)\cos(\varphi+\theta_2)+ \sin\alpha\cos\alpha \sin(\varphi+\theta_1)\sin(\varphi+\theta_2) \\
- \sin\alpha\cos\alpha \cos(\theta_1-\varphi)\cos(\theta_2-\varphi) - \sin\alpha\cos\alpha \sin(\theta_1-\varphi)\sin(\theta_2-\varphi)
\end{multline*}
which is $0$, and
\begin{multline*}
\partial_{\varphi(y)}(r_{\theta_1}x \cdot r_{\theta_2}y)\big|_{x=y=(\alpha,\varphi,0)}
=-\sin^2\alpha \cos(\varphi+\theta_1)\sin(\varphi+\theta_2)+ \sin^2\alpha \sin(\varphi+\theta_1)\cos(\varphi+\theta_2) \\
+ \cos^2\alpha \cos(\theta_1-\varphi)\sin(\theta_2-\varphi) - \cos^2\alpha \sin(\theta_1-\varphi)\cos(\theta_2-\varphi)
\end{multline*}
which simplifies to
\[
\cos (2\alpha) \sin (\theta_2-\theta_1).
\]
Therefore
\begin{align*}
\partial_{\alpha(y)}\Pi_N^m (x,y)\big|_{x=y}&=0\\
\partial_{\varphi(y)}\Pi_N^m (x,y)\big|_{x=y}&=-\frac{\cos (2\alpha)}{4\pi^2} \iint \exp (-im\theta_1 + im\theta_2) U_N'(\cos(\theta_1-\theta_2))\sin(\theta_1-\theta_2) d\theta_1 d\theta_2\\
&=-\frac{\cos (2\alpha)}{2\pi} \int \exp (-im\theta ) U_N'(\cos\theta) \sin \theta d\theta\\
&=-\frac{im\cos (2\alpha)}{2\pi} \int \exp (-im\theta ) U_N(\cos\theta) d\theta \\
&=-im\cos (2\alpha),
\end{align*}
by Lemma \ref{lem3}. Likewise,
\begin{multline*}
\partial_{\nu_1}\partial_{\nu_2}\Pi_N^m (x,y)\big|_{x=y} = \frac{1}{4\pi^2} \iint \exp (-im(\theta_1 -\theta_2)) \\
\times \left(U_N''(\cos(\theta_1-\theta_2))\partial_{\nu_1}(r_{\theta_1}x \cdot r_{\theta_2}y)\partial_{\nu_2}(r_{\theta_1}x \cdot r_{\theta_2}y)\big|_{x=y}+U_N'(\cos(\theta_1-\theta_2))\partial_{\nu_1}\partial_{\nu_2}(r_{\theta_1}x \cdot r_{\theta_2}y)\big|_{x=y}  \right) d\theta_1 d\theta_2,
\end{multline*}
and we have
\begin{align*}
\partial_{\alpha(x)}\partial_{\alpha(y)}(r_{\theta_1}x \cdot r_{\theta_2}y)\big|_{x=y=(\alpha,\varphi,0)} &= \cos(\theta_1-\theta_2)\\
\partial_{\alpha(x)}\partial_{\varphi(y)}(r_{\theta_1}x \cdot r_{\theta_2}y)\big|_{x=y=(\alpha,\varphi,0)} &= \cos\alpha \sin \alpha\sin(\theta_1-\theta_2)\\
\partial_{\varphi(x)}\partial_{\alpha(y)}(r_{\theta_1}x \cdot r_{\theta_2}y)\big|_{x=y=(\alpha,\varphi,0)} &= -\cos\alpha \sin \alpha \sin(\theta_1-\theta_2)\\
\partial_{\varphi(x)}\partial_{\varphi(y)}(r_{\theta_1}x \cdot r_{\theta_2}y)\big|_{x=y=(\alpha,\varphi,0)} &= \cos(\theta_1-\theta_2).
\end{align*}
Therefore
\begin{align*}
\partial_{\alpha(x)}\partial_{\alpha(y)}\Pi_N^m (x,y)\big|_{x=y=(\alpha,\varphi,0)} &= \frac{1}{2\pi} \int \cos (m\theta)  U_N'(\cos\theta) \cos\theta d\theta = \frac{N^2-m^2}{2}+N\\
\partial_{\alpha(x)}\partial_{\varphi(y)}\Pi_N^m (x,y)\big|_{x=y=(\alpha,\varphi,0)}&=im\cos\alpha \sin \alpha \\
\partial_{\varphi(x)}\partial_{\alpha(y)}\Pi_N^m (x,y)\big|_{x=y=(\alpha,\varphi,0)}&=-im\cos\alpha \sin \alpha .
\end{align*}
For the last case,
\begin{align*}
&\partial_{\varphi(x)}\partial_{\varphi(y)}\Pi_N^m (x,y)\big|_{x=y=(\alpha,\varphi,0)}\\
=&\frac{1}{2\pi} \int \cos (m\theta) \cos^2 (2\alpha)U_N''(\cos\theta)\sin^2 \theta d\theta +\frac{1}{2\pi} \int \cos (m\theta)   U_N'(\cos\theta) \cos\theta d\theta\\
=& \frac{\cos^2 (2\alpha)}{2\pi} \int \left(\cos (m\theta)\sin\theta\right)' U_N'(\cos\theta)d\theta +\frac{1}{2\pi} \int \cos (m\theta)   U_N'(\cos\theta) \cos\theta d\theta\\
=&-\frac{m \cos^2 (2\alpha) }{2\pi} \int  U_N'(\cos\theta)\sin (m\theta) \sin \theta d\theta +   \frac{1+\cos^2 2\alpha}{2\pi} \int \cos (m\theta)   U_N'(\cos\theta) \cos\theta d\theta\\
=&\frac{m^2 \cos^2 (2\alpha) }{2\pi} \int  U_N(\cos\theta)\cos (m\theta)  d\theta +   \frac{1+\cos^2 2\alpha}{2\pi} \int \cos (m\theta)   U_N'(\cos\theta) \cos\theta d\theta\\
=& m^2 \cos^2(2\alpha) + (1+\cos^2 (2\alpha))\left(\frac{N^2-m^2}{2}+N\right).
\end{align*}
Finally, recall that
\[
\Delta_N^m = \frac{1}{N+1}\begin{pmatrix} A& B \\ B^*& C\end{pmatrix},
\]
where
\[
A = \Pi_N^m(x,x), B = \begin{pmatrix}\mathbf{e}_1^y\Pi_N^m (x,y)\big|_{x=y} & \mathbf{e}_2^y\Pi_N^m (x,y)\big|_{x=y}\end{pmatrix}
\]
and
\[
C = \begin{pmatrix}\mathbf{e}_1^x\mathbf{e}_1^y\Pi_N^m (x,y) \big|_{x=y} &\mathbf{e}_1^x \mathbf{e}_2^y\Pi_N^m (x,y)\big|_{x=y}\\ \mathbf{e}_2^x\mathbf{e}_1^y\Pi_N^m (x,y)\big|_{x=y} & \mathbf{e}_2^x\mathbf{e}_2^y\Pi_N^m (x,y)\big|_{x=y}\end{pmatrix},
\]
where $\{\partial_\theta, \mathbf{e}_1,\mathbf{e}_2\}$ is a local orthonormal frame. (So $A$, $B$, and $C$ are $1 \times 1$, $1 \times 2$, and $2 \times 2$ complex matrices.) Above computation with $(\mathbf{e}_1,\mathbf{e}_2) = (\partial_\theta,\partial_\varphi)\big|_{(\alpha,\varphi)=(\pi/4,0)}$ imply that
\begin{multline*}
\Delta_N^m  = \frac{1}{N+1} \begin{pmatrix} 1 & 0 & -im \cos (2\alpha) \\ 0 &\frac{N^2-m^2}{2}+N & im\cos\alpha \sin \alpha \\im \cos (2\alpha) & -im\cos\alpha \sin \alpha & m^2 \cos^2(2\alpha) + (1+\cos^2 (2\alpha))\left(\frac{N^2-m^2}{2}+N\right) \end{pmatrix}\big|_{(\pi/4,0)}\\
=\frac{1}{N+1} \begin{pmatrix} 1 & 0 & 0 \\ 0 &\frac{N^2-m^2}{2}+N & \frac{im}{2} \\0 & -\frac{im}{2} & \frac{N^2-m^2}{2}+N \end{pmatrix}.
\end{multline*}
\end{proof}
\begin{proof}[Proof of Lemma \ref{KRLEM}]
From Proposition \ref{KRPROP} and Theorem \ref{cheb}, we see that
\[
\E(\#\{f_N^m =0\}) = \int_{\mathbb{C}^2} |\xi \wedge \bar{\xi}| \frac{N+1}{ \det \Lambda} \exp (-\langle\Lambda^{-1}\xi, \xi \rangle) dL(\xi) = \int_{\mathbb{C}^2} |\xi \wedge \bar{\xi}| \frac{N+1}{ \det \Lambda} \exp (-\langle\Lambda^{-1/2}\xi, \Lambda^{-1/2}\xi \rangle) dL(\xi),
\]
where
\[
\Lambda= \frac{1}{N+1}\begin{pmatrix}\frac{N^2-m^2}{2}+N & \frac{im}{2} \\ -\frac{im}{2} & \frac{N^2-m^2}{2}+N\end{pmatrix}.
\]
By change of variables $\Lambda^{-1/2}\xi = \zeta$, we have
\[
\E(\#\{f_N^m =0\}) = \frac{N+1}{\pi^3}\int_{\mathbb{C}^2} |\Lambda^{1/2}\zeta \wedge \overline{\Lambda^{1/2}\zeta}|  \exp (- |\zeta|^2) dL(\zeta).
\]
Now let
\[
\zeta = \begin{pmatrix}1&1\\ i & -i\end{pmatrix}\begin{pmatrix}\alpha_1\\ \alpha_2\end{pmatrix}
\]
and let $v_1 = \begin{pmatrix}1\\ i \end{pmatrix}$ and $v_2 = \begin{pmatrix}1\\ -i \end{pmatrix}$. Observe that $\Lambda v_1= (\mu+\nu)v_1$, and $\Lambda v_2 = (\mu-\nu)v_2$ with $\mu = \frac{1}{N+1}\left(\frac{N^2-m^2}{2}+N\right)$ and $\nu = \frac{m}{2(N+1)}$. In particular, we have
\begin{align*}
|\Lambda^{1/2}\zeta \wedge \overline{\Lambda^{1/2}\zeta}| &= |(\alpha_1 (\mu+\nu)^{\frac{1}{2}}v_1+ \alpha_2 (\mu-\nu)^{\frac{1}{2}} v_2)\wedge \overline{(\alpha_1 (\mu+\nu)^{\frac{1}{2}}v_1+ \alpha_2 (\mu-\nu)^{\frac{1}{2}} v_2)}|\\
&=\left||\alpha_1|^2 (\mu+\nu) v_1 \wedge \overline{v_1} + |\alpha_2|^2 (\mu-\nu) v_2 \wedge \overline{v_2}\right|\\
&= \mu\left||\alpha_1|^2 (1+\frac{\nu}{\mu})-|\alpha_2|^2 (1-\frac{\nu}{\mu})\right| |v_1 \wedge v_2|\\
&= 2\mu\left||\alpha_1|^2 (1+\eta)-|\alpha_2|^2 (1-\eta)\right|,
\end{align*}
where
\[
\eta = \frac{m/2}{\frac{N^2-m^2}{2}+N}.
\]
Therefore
\begin{align*}
\E(\#\{f_N^m =0\}) = \left(\frac{N^2-m^2}{2}+N\right)\frac{1}{\pi^3}\int_{\mathbb{C}^2} \left||\alpha_1|^2 (1+\eta)-|\alpha_2|^2 (1-\eta)\right| \exp (-2|\alpha|^2) dL(\alpha),
\end{align*}
and we evaluate the integral as follows:
\begin{align*}
&\frac{1}{\pi^3}\int_{\mathbb{C}^2} \left||\alpha_1|^2 (1+\eta)-|\alpha_2|^2 (1-\eta)\right| \exp (-2|\alpha|^2) dL(\alpha)\\
=&\frac{4}{\pi}\int_{\mathbb{R}_+^2} r_1r_2\left|r_1^2 (1+\eta)-r_2^2 (1-\eta)\right| \exp (-2(r_1^2+r_2^2)) dr_1dr_2 \tag{change of variables $\alpha_j=r_je^{i\theta_j}$}\\
=&\frac{4}{\pi}\int_0^{\pi/2}\int_{0}^\infty r^3 \cos\theta\sin\theta\left|r^2\cos^2\theta (1+\eta)-r^2\sin^2\theta (1-\eta)\right| \exp (-2r^2) drd\theta \tag{change of variables $(r_1,r_2)=(r\cos\theta,r\sin\theta)$}\\
=&\frac{4}{\pi}\int_{0}^\infty r^5 \exp (-2r^2) dr \int_0^{\pi/2}\cos\theta\sin\theta\left|\eta +\cos (2\theta) \right|  d\theta\\
=&\frac{1}{4\pi} \int_{-1}^{1}\left|\eta +t \right|  dt = \frac{1+\eta^2}{4\pi}.
\end{align*}
In the last equality, we used the fact that $|\eta| \leq \frac{1}{2}$ to evaluate the integral.
\end{proof}
\subsection{\label{BERTINISECT} Bertini theorem}
We need the following Bertini-type theorem to  employ the Kac-Rice formalism and to prove Theorem \ref{MAIN}:
\begin{prop} \label{BERTINI} For all $m \not= 0$, $0$ is almost surely a  regular value of the random equivariant eigenfunction
$\psi_N^m : \Ss^3 \to \C$. \end{prop}


\begin{proof} We need to show that the derivative $d_x \psi_N^m : T_x \Ss^3 \to \C$ is surjective at each point $x$
where $\psi_N^m (x) = 0$ for almost any $\psi_N^m $.  It is sufficient to prove that $\hcal_N^m$ has the ``1-jet spanning property''
that the 1-jet evaluation map,
$$\jcal: \Ss^3 \times \hcal_N^m \to J^1(\Ss^3, \C), \;\; \jcal(x, \psi_N^m ) =J^1_x \psi_N^m  =  (\psi_N^m (x), d_x^H \psi_N^m (x)), $$
is surjective.

The $1$-jet spanning property  implies
that at each point, $\{d_x^H \psi_N^m (x): \psi_N^m  \in \hcal_N^m\} $ spans  the horizontal tangent space $H_x \Ss^3$. Since the ensemble is $SU(2)$ invariant, it suffices to prove the spanning property at a single point. Moreover, since $SU(2)$ acts transitively on the unit tangent
bundle of $\Ss^2$ (or on the horizontal spaces of $\Ss^3$), failure to span is equivalent to the existence of $x$ such that  $d_x^H \psi_N^m (x) = 0$ for all $\psi_N^m $ such that $\psi_N^m (x)= 0$.  This is false, since Jacobi polynomials have simple zeros, as  can for instance be seen from the Darboux formula.


As explained in \cite[Section 4.1]{BSZ01},
the $1$-jet spanning property implies that  the incidence set
$I:=\{(x,\psi_N^m )\in \Ss^3\times \hcal_N^m: \psi_N^m (x)=0\}$ is a smooth submanifold,  and hence by
Sard's theorem applied to the projection $I\to \hcal_N^m$, the zero set
$$Z_{\psi_N^m }=\{x\in \Ss^3: \psi_N^m (x) =0\}$$ is a smooth $1$-dimensional submanifold of $\Ss^3$
for almost all $\psi_N^m $.
\end{proof}

\section{\label{COMPLETION} R\'esum\'e of the proof Theorem \ref{MAIN}}

Having established all  of the ingredients of the proof of Theorem \ref{MAIN} outlined in the Introduction, we only review how to assemble the ingredients into
a proof.  We resume the discussion begun in Section \ref{JZREVIEW}.  To prove (i) of Theorem \ref{MAIN} we use the proof of the same statement
as Theorem \ref{JZT}  (\cite[Theorem 1.5]{JZ18}) for general Kaluza-Klein Laplacians on circle bundles. The only point we need to establish to apply the proof is that $0$ is almost surely a regular  value of the
equivariant eigenfunctions, and this is proved in Theorem \ref{BERTINI}. The main new result is therefore (ii) of Theorem \ref{MAIN}. It is based on a simple
formula of Lemma \ref{CHILEM} relating the genus of the nodal set with the number of zeros of $f_N^m$; this makes use of the special structure of nodal set of real parts of equivariant
eigenfunctions as `helicoid covers' of $\Ss^2$.  In Section \ref{KRFSECT}, we  use the  Kac-Rice to calculate the expected number
of zeros of the random $f_N^m$ and prove  Lemma \ref{KRLEM}, concluding the proof of Theorem \ref{MAIN}.

\bibliography{bibfile}
\bibliographystyle{alpha}

\end{document}